\documentclass[letterpaper,notitlepage,12]{article}
\usepackage{setspace}
\usepackage{sectsty}
\usepackage{makeidx}
\usepackage{titling}
\usepackage{stmaryrd}
\usepackage{latexsym}
\usepackage{amsbsy}
\usepackage{amssymb}
\usepackage{verbatim}
\usepackage{amsthm}
\usepackage{amsfonts}
\usepackage{amsmath}
\usepackage{graphicx}
\usepackage{layout}
\usepackage{pb-diagram}
\usepackage{amscd}
\usepackage{anysize}
\usepackage{tikz}
\usepackage{etoolbox}
\usepackage{graphicx}
\usepackage[all,cmtip]{xy}
\usepackage{caption}
\usepackage{hyperref}
\newcommand{\cita}{\textquotedblleft}

\newcommand{\Z}{\mathbb{Z}}

\newcommand{\R}{\mathbb{R}}

\newcommand{\D}{\mathbb{D}}

\newcommand{\nbeq}{\begin{equation}}
\newcommand{\neeq}{\end{equation}}
\newcommand{\beq}{\begin{equation*}}
\newcommand{\eeq}{\end{equation*}}
\theoremstyle{definition}
\newtheorem{thm}{Theorem}[section]% reset theorem numbering for each section
\newtheorem{lemma}[thm]{Lemma}% lemma numbers are dependent on theorem numbers
\newtheorem{defin}[thm]{Definition}
\newtheorem{propo}[thm]{Proposition}
\newtheorem{coro}[thm]{Corollary}
\newtheorem{rmk}{Remark}
\newtheorem{add}{Addendum}

\newtheorem*{claim}{Claim}

\newcommand{\bproof}{\begin{proof}}
\newcommand{\eproof}{\end{proof}}
\marginsize{3cm}{3cm}{3cm}{3cm} %left  right  up  down
\sectionfont{\small}
\allsectionsfont{\centering}
\sectionfont{\fontsize{11}{15}\selectfont}
\title{\Large\bf On the topology of the space of pinched negatively curved
metrics with finite volume and identical ends}
\author{Mauricio Bustamante}
\newcommand{\Addresses}{{% additional braces for segregating \footnotesize
  \bigskip
  \footnotesize
  \textsc{Mauricio Bustamante}, \textsc{Department of Mathematical Sciences, Binghamton University, Binghamton, NY}\par\nopagebreak
 \texttt{bustamante.math@gmail.com}
  }}
\date{}
\begin{document}
\maketitle
\vspace{-1.2cm}
\begin{abstract}
We prove that the space of complete, finite volume,
pinched, negatively curved Riemannian metrics on a smooth high-dimensional
manifold is either empty or it is highly non-connected, provided their behavior at infinity is similar.
\end{abstract}
%%%%%%%%%%%%%%%%%%%%%%%%%%%%%%%%%%%%%%%%%%%%%%%%%%%%%%%%%%%%%%%5
%%%%%%%%%%%%%%%%%%%%%%%%%%%%%%%%%%%%%%%%%%%%%%%%%%%%%%%%%%%%%%%%
\section{Introduction}
\par Let $M$ be a connected, noncompact smooth manifold. 
We denote by $MET(M)$ the space of all Riemannian metrics on $M$, 
endowed with the compact-open $C^{\infty}$-topology. We let
$MET^{<0}(M)$ denote the subspace of $MET(M)$ consisting of all 
complete Riemannian metrics on $M$ with finite volume and whose sectional curvatures are all bounded by two
negative numbers (in that case we say that a metric
is pinched negatively curved).

Before we state our main theorem we recall the definition
of an \textit{end} of a noncompact space.
\begin{defin}\label{end}
An \textit{end} of a noncompact space $X$ is a function $E$ that assigns to
each compact subset $K$ of $M$, a connected component $E(K)$ of $X-K$, subject
to the condition that if $K\subset K'$ then $E(K')\subset E(K)$.\\
An open set $U\subset X$ is a \textit{neighborhood of an end} 
$E$ if there exists 
a compact set $K$ in $X$ such that $E(K)\subset U$.
\end{defin}
%\begin{defin}\label{identical ends}
%Two Riemannian metrics $g_0,g_1\in MET^{<0}(M)$ are said to have
%\textit{identical ends} if for every end $E$ of $M$ there exists a 
%neighborhood $U$ of $E$ such that 
%\beq
%id:(U,g_0|_U)\to (U,g_1|_U)
%\eeq
%is an isometry, where $g_i|_U$ ($i=0,1$) is the 
%restriction of $g_i$ to $U$.
%\end{defin}

Assume that $MET^{<0}(M)\neq\emptyset$. For a compact subset $K\subset M$,
and $g\in MET^{<0}(M)$, let $MET^{<0}_{K}(M,g)$ be the
subspace of $MET^{<0}(M)$ consisting of all Riemannian metrics
$g'\in MET^{<0}(M)$
such that $g'$ and $g$ \textit{agree} on $M-K$, i.e. the restriction of the 
identity map $id:M\to M$, to $M-K$
\beq
id|_{M-K}:(M-K,g|_{M-K})\to (M-K,g'|_{M-K}),
\eeq
is an isometry, where $g|_{M-K}$  denotes the 
restriction of $g$ to $M-K$.

Let 
$K_1\subset K_2\subset\cdots$ be a sequence of compact 
subsets covering $M$. It yields a sequence of subspaces
\beq
MET^{<0}_{K_1}(M,g)\subset MET^{<0}_{K_2}(M,g)\subset\cdots
\eeq

Its union $\displaystyle\bigcup_{i=1}^{\infty}MET^{<0}_{K_i}(M,g)$ is
given the direct limit topology and is denoted by
$MET_{\infty}^{<0}(M,g)$,
i.e. a subset $A$ in
$\displaystyle MET^{<0}_{\infty}(M,g)$
is closed if and only if $A\bigcap MET_{K_n}^{<0}(M,g)$ is closed in
$MET_{K_n}^{<0}(M,g)$ for each $n$.

It will also be useful to recall that a map $f:MET^{<0}_{\infty}(M,g)\to X$
from the direct limit $MET^{<0}_{\infty}(M,g)$ to a space $X$ is 
continuous if and only if for all $i=1,2,\ldots$, $f$ restricted to
$MET_{K_i}^{<0}(M,g)$ is continuous. (see \cite[p. 5, 18]{spanier}).
\begin{lemma}
The definition of $MET_{\infty}^{<0}(M,g)$ is independent of the choice
of the sequence of compact subsets covering $M$.
\end{lemma}
\begin{proof}
Let
$L_1\subset L_2\subset\cdots$ and $K_1\subset K_2\subset\cdots$ be
two sequences of compact subsets
covering $M$, then it is clear that 
$\bigcup MET^{<0}_{K_i}(M,g)=\bigcup MET^{<0}_{L_i}(M,g)$ as sets. Give
them the direct limit topology.

Note that every $K_i$ is contained in some $L_j$ and 
$MET^{<0}_{K_{i}}(M,g)$ is a subspace of $MET^{<0}_{L_{j}}(M,g)$. Thus,
for all $i=1,2,\ldots,$ we have a
continuous map
\beq
MET^{<0}_{K_{i}}(M,g)\hookrightarrow 
\bigcup_{j=1}^{\infty} MET^{<0}_{L_j}(M,g).
\eeq
Hence
the identity map
\beq
id:\bigcup_{i=1}^{\infty}MET^{<0}_{K_i}(M,g)\to\bigcup_{j=1}^{\infty}MET^{<0}_{L_j}(M,g)
\eeq
is continuous. Similarly one proves that the identity map going in the
opposite direction is also continuous. This proves the lemma.
\end{proof}

In this paper we want to make use of some results in pseudoisotopy
theory in order
to shed light on the topology of $MET^{<0}_{\infty}(M,g)$. 
Let $(\Z/p)^{\infty}$ denote a countably infinite sum of finite 
cyclic groups of order $p$. We prove the following theorem.
\begin{thm}\label{main1}
Let $M$ be a noncompact manifold and assume that $MET^{<0}(M)$
is nonempty. Then, for any $g\in MET^{<0}(M)$
\begin{enumerate}
\item $MET_{\infty}^{<0}(M,g)$ has infinitely many
path connected components, provided $dim\, M\geq 10$.
\item If $\mathcal{K}$ is any component 
of $MET^{<0}_{\infty}(M,g)$, then $\pi_1(\mathcal{K})$
contains a subgroup isomorphic to $(\Z/2)^{\infty}$, provided $dim\, M~\geq~ 14$, and
\item For each odd prime $p$, 
$\pi_{2p-4}(\mathcal{K})$ contains a subgroup isomorphic to $(\Z/p)^{\infty}$, provided
$\frac{dim\, M-10}{2}> 2p-4$.
\end{enumerate}
\end{thm}

These results can be considered as an extension to noncompact
manifolds of those obtained by Farrell
and Ontaneda in \cite{farrell2010} for high-dimensional compact manifolds
with negative sectional curvature. (Compare also \cite{FO15} where spaces of
\textit{nonpositively} curved metrics on negatively curved manifolds are considered.)
\par In fact, we show that the same basic idea of their paper can be
extended to the setting of noncompact manifolds provided the metrics
agree off a sufficiently large compact set.
We establish, in Section \ref{preliminaries}, the existence
of closed geodesics in some complete negatively curved manifolds and prove that certain family of closed geodesics, arising from continuous variations of negatively curved  Riemannian metrics, varies continuously as well.
In Section \ref{components} we prove Theorem \ref{main1}.

This paper is part of my thesis \cite{bustamante} written under
the supervision of Tom Farrell at Binghamton University. I am
infinitely grateful to him for his guidance 
during this project and for pointing out key ideas 
that led to the results of this paper.
I also benefited from conversations with Andrey Gogolev, Tam Nguyen-Phan and Pedro Ontaneda. I acknowledge the referee, whose thorough reading
of the paper led to a better exposition of the results. Finally
I acknowledge the hospitality of Yau Mathematical Sciences Center of
Tsinghua University in Beijing, where part of this project was carried
out.
%%%%%%%%%%%%%%%%%%%%%%%%%%%%%%%%%%%%%%%%%%%%%%%%%%%%%%%%%%%%%%%%%%%%%%
%%%%%%%%%%%%%%%%%%%%%%%%%%%%%%%%%%%%%%%%%%%%%%%%%%%%%%%%%%%%%%%%%%%%%%
\section{Closed geodesics in noncompact negatively curved manifolds}\label{preliminaries}
In this section we recall some basic facts about the topology of 
noncompact manifolds that admit a complete, pinched negatively curved
Riemannian metric with finite volume. We establish the existence
of closed geodesics representing certain free homotopy classes of
loops and analyze the way they bahave under variations of the metric.
%One of the crucial results to understand the topology of negatively
%curved manifolds is the Margulis Lemma:
%\begin{thm}[(Margulis Lemma)] 
%There is a constant $\mu=\mu(n)>0$ depending only on $n$ with the following
%property: let $X$ be a complete, simply connected, $n$-dimensional
%manifold with curvature $-1\leq K\leq 0$, let $\Gamma$ be a discrete group of
%isometries of $X$ and $x\in X$. Then the group 
%\beq
%\Gamma_{\mu}(x)=\langle \gamma\in\Gamma|d(x,\gamma x)\leq\mu\rangle
%\eeq
%is virtually nilpotent.
%\end{thm}

\par For a complete manifold $(M,g)$ with pinched negative curvature and finite volume, the $\epsilon$-\textit{thin} part of $M$ is
\beq
M_{<\epsilon}=\{x\in M|\text{InjRad}_g(x)<\epsilon\}.
\eeq

Let $0<\epsilon\leq\mu/2$ ($\mu$ the Margulis constant
\cite[p. 101]{BalGro85}). Each 
end $E$ of $M$ can be realized as a connected
component of $M_{<\epsilon}$, that is, there is a unique connected component
$U_{\epsilon}(E)$ of $M_{<\epsilon}$ such that $U_{\epsilon}(E)$ is a 
neighborhood of $E$. For every end $E$ there exists a codimension $1$ closed submanifold $N$ of $M$ such that $U_{\epsilon}(E)$ is diffeomorphic to
$N\times(0,\infty)$. Each $N$ is called a \textit{cusp cross section} and
its fundamental group is isomorphic to a maximal
virtually nilpotent subgroup of $\pi_1M$.
Furthermore, $M$ has only finitely many ends and
any two ends of $M$ have disjoint neigborhoods. These
facts are proven in \cite{gromov1978}, \cite{Eberlein1980},\cite{schroeder1984}.
\par The fundamental group $\pi_1M$ of $M$ acts freely, discretely and by
isometries on the universal cover $\widetilde{M}$ of $M$. The elements of
$\pi_1M$ are then classified into two types: either they are \textit{hyperbolic}
if they fix exactly two points 
of the boundary at infinity of $\widetilde{M}$ and translates the unique geodesic
joining these points, or \textit{parabolic} if they fix
exactly one point $\bar{x}$ of the boundary at infinity and leave the 
horospheres at $\bar{x}$ invariant. The boundary at infinity is 
understood here as equivalence classes of asymptotic geodesics in $\widetilde{M}$
(see \cite{eberlein1973}). Another characterization of hyperbolic
isometries is given in terms of the \textit{displacement function}. For
each isometry $f:\widetilde{M}\to\widetilde{M}$, the displacement function
$d_{f}:\widetilde{M}\to\R$ takes $x\in\widetilde{M}$ to $d(x,f(x))$. $f$ is hyperbolic if and only if it has no fixed points and 
$d_f$ attains a minimum (see \cite{BalGro85}).

We now prove that certain free homotopy classes of loops in $M$ can be
represented by closed geodesics in $M$. The first step towards this is
to guarantee the existence of closed geodesics in $M$. This is a
consequence of the Closing Lemma of Ballman-Brin-Spatzier:
\begin{thm}[(Ballman-Brin-Spatzier)]\label{closing lemma}
Let $M$ be a complete Riemannian manifold with finite volume
and nonpositive sectional curvature bounded from below. 
Then the set of vectors
tangent to regular closed geodesics is dense in the unit tangent
bundle $SM$ of $M$.
\end{thm}
For a proof refer to \cite[Corollary 4.6]{ballmann85}.\\\\
Let $(M,g)$ be a noncompact, complete Riemannian manifold with curvature
$-1\leq K\leq -a^2<0$ and finite volume. Suppose that $M$ has
ends $E_1,\ldots, E_r$ with pairwise disjoint neighborhoods $U_1,\ldots, U_r$
respectively, where $U_i$ is diffeomorphic to $N_i\times (0,\infty)$
for $i=1,\ldots, r$ ($N_i$ are the cusp cross sections of $M$).

Let $[S^1,M]$ denote the set of free homotopy classes of loops, i.e.
homotopy classes of continuous maps $S^1\to M$. For every 
$i=1,\ldots, r$, there is a well-defined
map $[S^1,U_i]\to [S^1,M]$ which sends a class $[\alpha]\in [S^1,U_i]$
to $[\sigma_i\circ\alpha]\in [S^1,M]$, where 
$\sigma_i:U_i\hookrightarrow M$ is the inclusion map. Let 
$F_i\subset [S^1,M]$ denote the image of $[S^1,U_i]$ in $[S^1,M]$.
Define the set $\textbf{H}:=[S^1,M]-\displaystyle\bigcup_i F_i$.
\begin{propo}\label{existence}
The set $\textbf{H}$ is nonempty. Moreover every class in $\textbf{H}$
can be represented by a unique closed geodesic in $M$.
\end{propo}
\begin{proof}
By Theorem \ref{closing lemma}, there exists a closed geodesic
$\alpha:S^1\to M$ in $M$. We claim that $\alpha$
is not freely homotopic to a loop whose image is contained in $U_i$, for 
any $i$.
To see this, suppose that $\beta:S^1\to M$ is a loop in $M$ freely homotopic
to $\alpha$ and $\beta(S^1)\subset U_i$, for some $i$.\\
Since $\alpha$ is a closed geodesic in $M$, then any lifting $\widetilde{\alpha}$
is contained in a unique geodesic line in the universal cover $\widetilde{M}$ that
joins two points of the boundary at infinity of $\widetilde{M}$. Therefore 
its class in $\pi_1(M,\alpha(*))$, $*\in S^1$, 
can be regarded as a hyperbolic isometry of 
$\widetilde{M}$. On the other hand, the loop $\beta$ can be \cita pushed towards
infinity'' which only means that its class in 
$\pi_1(U_i,\beta(*))\subset\pi_1(M,\beta(*))$ 
can be regarded as a parabolic isometry
of $\widetilde{M}$. (We see $\pi_1(U_i)$ as a subgroup of 
$\pi_1(M)$ using the fact that $\pi_1(U_i)\simeq\pi_1(N_i)$,
where $N_i$ is the cusp cross section of $E_i$).\\
Since a hyperbolic isometry can't be conjugated to a
parabolic isometry, the loops $\alpha$ and $\beta$ can't be freely homotopic. This
proves that $\textbf{H}$ is nonempty, which is the first part of the proposition.

Now suppose that a class $[c]\in\textbf{H}$ is given. Then
$c:I\to M$ represents a class in $\pi_1(M,c(0))$ which 
corresponds to a deck transformation  $f_c$ of $\widetilde{M}$. This deck
transformation is a hyperbolic isometry of $\widetilde{M}$. Let
$\gamma:\R\to\widetilde{M}$ be the unique maximal geodesic in $\widetilde{M}$
translated by $f_c$.

Recall that the image of $\gamma$ is a convex subset of 
$\widetilde{M}$. Hence for each $s\in I$, there exists a unique point
$p_s\in\gamma(\R)$,
called the orthogonal projection onto $\gamma(\R)$,
such that 
\beq
d(\tilde{c}(s),\gamma(\R))=d(\tilde{c}(s),p_s).
\eeq

Let $p_0\in\gamma(\R)$
be the orthogonal projection of $\tilde{c}(0)$ onto $\gamma(\R)$, where
$\tilde{c}:I\to\widetilde{M}$ a lifting of $c$ to $\widetilde{M}$. It is
not hard to see that $p_1:=f_c(p_0)$ is the orthogonal projection of 
$\tilde{c}(1)$ onto $\gamma(\R)$. For otherwise, if $p'$ is the orthogonal
projection, then the sum of interior angles of the geodesic triangle with
vertices $p', p_1$ and $\tilde{c}(1)$ would be greater than $\pi$.

Let $\gamma_{p_0,p_1}\subset\widetilde{M}$ be the geodesic segment from $p_0$ to
$p_1$.
%We claim that $p_0=y=\gamma(0)$. To see this note that
%\begin{align*}
%d_{f_c}(p_0)&=d(p_0,f_c(p_0))\\
%			 &\leq d(p_0,f_c(y))+d(f_c(y),f_c(p_0))\\
%			 &=d(p_0,f_c(y))+d(y,p_0)\\
%			 &=d(y,f_c(y))\\
%			 &=d_{f_c}(y),
%\end{align*}
%and since $y$ is where $d_{f_c}$ attains the minimum, the claim follows.
%Also notice that $p_1=f_c(y)$
Now, for each $s\in I$, let $\sigma_s:I\to\widetilde{M}$ be the unique
geodesic segment in $\widetilde{M}$ joining $\tilde{c}(s)$ with the orthogonal
projection of $\tilde{c}(s)$ onto the closed convex subset 
$\gamma_{p_0,p_1}\subset\widetilde{M}$ (see Figure 1).

\begin{figure}[!h]\label{homotopy}
\vspace*{0cm}
    \begin{center}
    \includegraphics[scale=0.6]{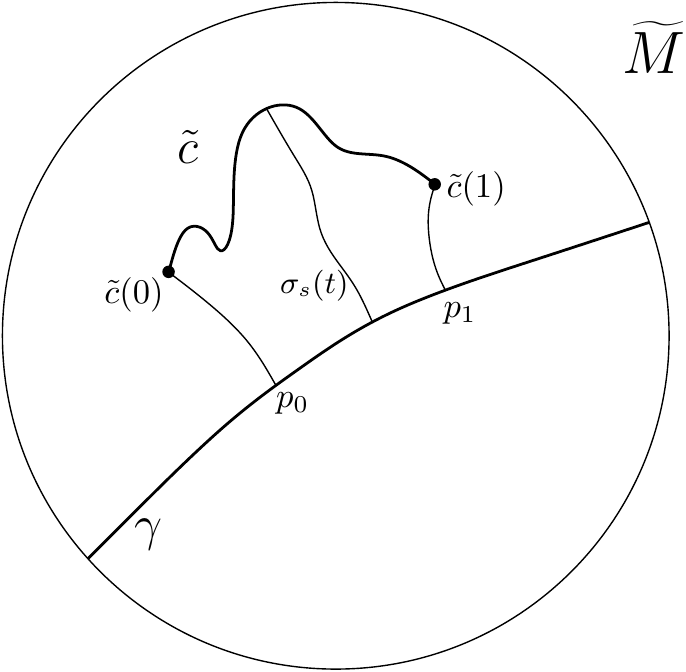}
    \caption{\small Equivariant homotopy between $\tilde{c}$ and
    a geodesic segment.}
    \end{center}
 \end{figure}

Note that $(f_c\circ\sigma_0)(t)=\sigma_1(t)$. Hence
the map $I\times I\to\widetilde{M}$
sending $(s,t)$ to $\sigma_s(t)$ is a continuous 
homotopy from $\tilde{c}$ to $\gamma$ which is
equivariant respect to deck transformations. Therefore this homotopy
descends to a homotopy $F_t:I\to M$, $t\in [0,1]$, such that
$F_0(s)=c(s)$ and $F_1(s)$ is a closed geodesic in $M$. 

Finally, if there is another closed geodesic $\beta$ 
freely homotopic to $c$ (and therefore to $F_1$), then the
liftings of the geodesics $\beta$ and $F_1$ to the universal cover of $M$ 
are contained in the same maximal geodesic in $\widetilde{M}$.
Hence the images of $\beta$ and $F_1$ in $M$ must coincide. 
This completes the proof of the proposition.
\end{proof}
%Note that the proof of Proposition \ref{existence} yields the following addendum
%\begin{coro}
%Let $(M,g)$ be a noncompact, complete Riemannian manifold with curvature
%$-1\leq K\leq -a^2<0$ and finite volume. Let $N$ be a component of the
%cusp cross section of $M$. Then $\pi_1N$ is \textit{not} isomorphic to
%$\pi_1M$. Furthermore, the index of $\pi_1N$ in $\pi_1M$ is infinite.
%\end{coro}
%\begin{proof}
%From the the proof of the proposition it is not hard to see that
%$[S^1,M]-[S^1,N]$ is nonempty.
%If $\pi_1N\simeq\pi_1M$ then  $[S^1,N]=[S^1,M]$, which is a contradiction.
%To see that the index is infinite we suppose that it is finite. Let
%$\widehat{M}$ be the finite sheeted covering space of $M$ corresponding to
%$\pi_1N$. Hence, $\widehat{M}$ has finite volume. On the other hand 
%$\widehat{M}$ is isometric to a quotient of the universal cover $\widetilde{M}$
%of $M$ by the action of the virtually nilpotent group $\pi_1N$. Thus it
%has infinite volume. This contradiction proves the corollary.
%\end{proof}
\begin{coro}\label{normalbundle}
Let $(M,g)$ be a noncompact, complete Riemannian manifold with curvature
$-1\leq K\leq -a^2<0$ and finite volume. Then there exists a free homotopy class
of loops $[c]\in [S^1,M]$ such that $[c]$ is represented by a closed geodesic
in $M$ with trivial normal bundle.
\end{coro}
\begin{proof}
Let $[c]\in\textbf{H}$. Suppose that the representative $c:[0,1]\to M$ of this
class is an immersed closed curve.
Denote by $\bar{c}\in\pi_1(M,c(0))$ its
fundamental group class. Since $\pi_1M$ is torsion free, the kernel of
the first Stiefel-Whitney class of the normal bundle of $c$ in $M$ is non-trivial.
Thus, without lost of generality we can assume that 
$\bar{c}\in\ker (w_1:\pi_1M\to\Z/2)$. For otherwise
we just take \cita twice $\bar{c}$\, ''. Therefore, by Proposition
\ref{existence}, $[c]$ can be represented
by a closed geodesic and the normal bundle of this closed geodesic has vanishing
first Stiefel-Whitney class.
\end{proof}
We now analyze the behavior of those closed geodesics under perturbations of the
metric. For this we need to recall some results from the theory of dynamical
systems. Let $X$ be a smooth manifold and dentote by 
%$\Gamma^r(TX)$ 
$\Gamma(TX)$ the space
of smooth vector fields on $X$ with the compact-open $C^{\infty}$-topology.
%($r\geq 1$).
Let $|\gamma|$ be a periodic orbit
%\footnote{An orbit $|\gamma|$ is understood here as the image
%in $X$ of an integral curve $\gamma:\R\to X$ of a vector field. Thus,
%$|\gamma|$
%is a periodic orbit if it is an immersed circle in $X$.} 
of a vector field $\xi\in\Gamma(TX)$ 
%$\xi\in\Gamma^r(TX)$ 
and let $\Sigma\subset X$ be a section transversal
to $\xi$ through a point $x_0\in|\gamma|\subset X$. 
(An orbit $|\gamma|$ of a vector field is understood here as the image
in $X$ of an integral curve $\gamma:\R\to X$ of the vector field).
For a small neighborhood
$V\subset\Sigma$ of $x_0$, let $P:V\to\Sigma$ be the $C^{\infty}$-map 
that assigns to each $x\in V$
the first point where the orbit through $x$ returns to intersect $\Sigma$.
$P$ is 
called the \textit{Poincar\'e map} associated to the vector field $\xi$
or to the closed orbit $\gamma$.
Notice that periodic
orbits of $\xi$ through points in $V$ correspond to fixed points of the 
Poincar\'e map. We say that a fixed point $x\in V$ of the Poincar\'e map
$P:V\to\Sigma$ is \textit{hyperbolic} if the derivative of $P$ at $x$ has no eingevalues of 
modulus $1$. A periodic orbit through a hyperbolic point is called a 
\textit{hyperbolic periodic orbit}.

Recall that if $(X,g)$ is
a smooth Riemannian manifold,
its \textit{unit tangent bundle} is defined
by $S_gX=\{v\in TX|\sqrt{g(v,v)}=1\}$. There is 
a unique vector field $G:S_gX\to TS_gX$ on the unit 
tangent bundle of $X$, whose integral
curves in $S_gX$ are of the form $t\mapsto(\alpha(t),\dot{\alpha}(t))$, 
where $\alpha(t)$ is a unit speed geodesic in $(X,g)$. 
The vector field $G$ is called the
\textit{$g$-geodesic field} and its flow the \textit{$g$-geodesic flow} on $S_gX$.
Periodic orbits of the $g$-geodesic flow on $SX$ are in
one-to-one correspondence with closed geodesics in $(X,g)$ (up to reparametrization). 

It will be convenient to refer to the geodesic flow on the \textit{sphere bundle}
$SX$ of $X$. $SX$ is defined as the quotient of $TX-X$ by identifying
two non-zero tangent vectors if they lie
on the same ray, that is, $v,w\in TX-X$ are equivalent 
if they are based at the same point and
if there is a positive number $a$ such that $w=a\, v$.
Note that the map $\eta:SX\to S_gX$ given
by $v\mapsto v/\sqrt{g(v,v)}$ defines a diffeomorphism for any Riemannian metric $g$. The
geodesic field on the sphere bundle is defined by the composition
$SX\xrightarrow{\eta}S_gX\xrightarrow{G} TS_gX\xrightarrow{\eta_*^{-1}}TSX$,
where $\eta_*$ denotes the differential map of $\eta$. Its flow will be
called the \textit{$g$-geodesic flow on the sphere bundle} $SX$.
Note that if $|\gamma|$ is a hyperbolic periodic orbit of the geodesic flow on the
unit tangent bundle $S_gX$ then
$\eta^{-1}(|\gamma|)$ 
is a hyperbolic periodic orbit of the geodesic flow on the sphere bundle $SX$.

For the rest of this section we work in the following setting:
$M$ will be a smooth noncompact
manifold such that $MET^{<0}(M)$ is nonempty. We consider a continuous
family of Riemannian metrics $g_s\in MET^{<0}(M)$, $s\in D^{k+1}:=\{x\in\R^{k+1}||x|\leq 1\}$.
%It induces a continuous $(k+1)$-parameter family of geodesic fields on $SM$. 
We fix a free homotopy class of loops $[c]\in \textbf{H}$. 
By Proposition \ref{existence}, there
exists a unique
periodic orbit $|\gamma_s|$ of the $g_s$-geodesic flow on $SM$ such that 
$[\tau_M\circ\gamma_s]=[c]$, where $\tau_M:SM\to M$ denotes the canonical
projection onto $M$. Let $C^{\infty}(S^1,M)$ be the space of smooth 
maps from the circle to $M$, with the compact-open-$C^{\infty}$ topology.

%%%%%%%%%%%%%%%%%%%%%%%%%%%%%%%%%%%%%%%%%%%%%
%%%%%%%%%%%%%%%%%%%%%%%%%%%%%%%%%%%%%%%%%%%%%%%%%%%%%%%%%%%%%
%%%%%%%%%%%%%%%%%%%%%%%%%%%%%%%%%%%%%%%%%%%%%%%%%%%%%%%%%%%%%
The next lemma is well known and can be proven in
more generality for Poincar\'e maps associated to complete vector 
fields on smooth manifolds. Since our
interest here is only in the geodesic field on $SM$, 
we state and sketch the proof of the result for this case only.
\begin{lemma}\label{poincare}
Let $\Sigma\subset SM$ be a section transversal to the $g_{s_0}$ geodesic field on $SM$ 
through a point $v_{s_0}\in |\gamma_{s_0}|$. 
Then there exist
$\varepsilon>0$ and an open neighborhood $V\subset\Sigma$ of
$v_{s_0}$ such that, for $s\in B_{\varepsilon}(s_0)$, the Poincar\'e map
$P_s$ associated to $\gamma_s$ is defined and the map 
$B_{\varepsilon}(s_0)\to C^{\infty}(V,\Sigma)$, 
sending $s\in B_{\varepsilon}(s_0)$ to $P_s$ is continuous.
\end{lemma}
\begin{proof}
The proof is just a parametrized version of 
\cite[Lemma 3.1.10]{klingenberg1978}. Let $G_s:SM\times\R\to SM$ denote
the $g_s$-geodesic flow on $SM$. Define a map 
\beq
\psi:D^{k+1}\times\R\times\Sigma\to
D^{k+1}\times SM
\eeq
by
\beq
\psi(s,t,v)=(s,G_s(t,v)).
\eeq
This is a smooth map since $G_s$ varies continuously with $s$. Using
the fact that $G_s(\ \ ,t):SM\to SM$ is a diffeomorphism, one can check that
the derivative $D\psi$ at $(s_0,\tau_{s_0},v_{s_0})$ is invertible. 
Here $\tau_{s_0}$ denotes the prime period of $\gamma_{s_0}$. Hence,
by the inverse function theorem, there exist $\varepsilon,\varepsilon'>0$
and an open neighborhood $V_{0}\subset\Sigma$ of $v_{s_0}$ so that the restriction
of $\psi$ to 
$B_{\varepsilon}(s_0)\times (\tau_{s_0}-\varepsilon',\tau_{s_0}+\varepsilon')\times V_0\subset D^{k+1}\times\R\times\Sigma$
is a diffeomorphism onto its image. Thus there is a neighborhood
$U\subset SM$ of $v_{s_0}$ such that $\psi^{-1}$ is defined on
$B_{\varepsilon}(s_0)\times U\subset D^{k+1}\times SM$. In particular,
we have 
smooth mappings
\beq
\chi:B_{\varepsilon}(s_0)\times U\to
V_0\ \ \ \text{ and }\ \ \ 
\eta:B_{\varepsilon}(s_0)\times U\to\R,
\eeq
such that, for
every $(s,v)\in B_{\varepsilon}(s_0)\times U$, 
\beq
\psi(s,\tau_{s_0}+\eta(s,v),\chi(s,v))=(s,v).
\eeq

Denote by $\widehat{\chi}:B_{\varepsilon}(s_0)\times U\to
B_{\varepsilon}(s_0)\times V_0$ the map 
$\widehat{\chi}(s,v)=(s,\chi(s,v))$. The restriction of 
$\widehat{\chi}$ to $B_{\varepsilon}(s_0)\times V_0$ is a diffeomorphism
onto its image $\mbox{im}(\widehat{\chi})=B_{\varepsilon}(s_0)\times V$
where $V\subset\Sigma$ is some neighborhood of $v_{s_0}$. 

The Poincar\'e map $P_s$ associated with $\gamma_s$ is nothing but
\beq
P_s=pr\circ\widehat{\chi}^{-1}(s,\ \ )
\in C^{\infty}(V,\Sigma),
\eeq
where $pr:B_{\varepsilon}(s_0)\times V_0\to V_0$ denotes projection onto the second component.
%Thus we obtain a continuous map
%\beq
%\Psi:B_{\varepsilon}(s_0)\to C^{\infty}(V,B_{\varepsilon}(s_0)\times\Sigma)
%\eeq
%given by $\Psi(s)(v)=\widehat{\chi}^{-1}(s,v)$. If we denote by 
%$pr:C^{\infty}(V,B_{\varepsilon}(s_0)\times\Sigma)
%\to C^{\infty}(V,\Sigma)$  the map
%that sends a smooth map $f:V\to B_{\varepsilon}(s_0)\times\Sigma$ to its
%composition with
%the projection onto the second factor
%$V\xrightarrow{f} B_{\varepsilon}(s_0)\times SM\xrightarrow{pr_2}SM$, 
%then the
%Poincar\'e map $P_s$ can be written as 
%\beq
%P_s=pr(\Psi(s)).
%\eeq 

This completes the proof of the lemma.
\end{proof}
%%%%%%%%%%%%%%%%%%%%%%%%%%%%%%%%%%%%%%%%%%%%%%%%%%
%The proof of Lemma \ref{poincare} tells us more.
\begin{coro}\label{period}
The periods $\tau_s$ of the periodic orbits $\gamma_s$ of the
$g_s$-geodesic flow on $SM$ vary continuously on $B_{\varepsilon}(s_0)$.
\end{coro}
\begin{proof}
The proof of Lemma \ref{poincare} shows that, for 
$s\in B_{\varepsilon}(s_0)$,
$\tau_s=\tau_{s_0}+\eta(s,P_s(v_s))$, 
where $v_s\in |\gamma_s|\cap V$. 
\end{proof}
\begin{rmk}
Although the orbits $|\gamma_s|$ and their periods are close to each other,
they may be farther apart as parametrized curves. This is remedied
in the proof of the next proposition by linearly reparametrizing the geodesic fields.

\begin{figure}[!h]\label{closeeachother}
\vspace*{0cm}
    \begin{center}
    \includegraphics[scale=0.6]{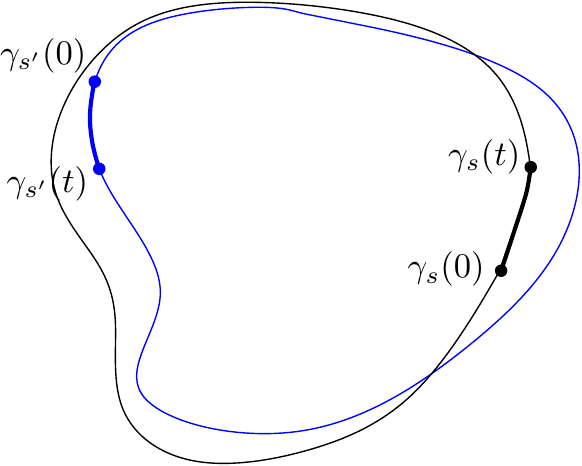}
    \caption{\small Two nearby
    periodic orbits of the $g_s$ and $g_{s'}$ geodesic flows which are \textit{not}
    close to each other as parametrized curves.}
    \end{center}
 \end{figure}

\end{rmk}
%%%%%%%%%%%%%%%%%%%%%%%%%%%%%%%%%%%%%%%%%%%%%
\begin{propo}\label{local argument}
Let $s_0\in D^{k+1}$. There exist $\varepsilon>0$ and
a family of parametrized closed smooth curves 
$\alpha_s:S^1\to M$, $s\in B_{\varepsilon}(s_0)$, such that 
\begin{itemize}
\item[i)] $[\alpha_s]=[c]$,
\item[ii)]the map $ B_{\varepsilon}(s_0)\to C^{\infty}(S^1,M)$ given by
$s\mapsto\alpha_s$ is continuous, and
\item[iii)] $\alpha_s$ is a closed $g_s$-geodesic in $M$.
\end{itemize}
\end{propo}
%%%%%%%%%%%%%%%%%%%%%%%%%%%%%%%%%%%%%%%%%%%%%%%%%%%%%%%%%%%%
%%%%%%%%%%%%%%%%%%%%%%%%%%%%%%%%%%%%%%%%%%%%%%%%%%%%%%%%
%%%%%%%%%%%%%%%%%%%%%%%%%%%%%%%%%%%%%%%%%%%%%%%%%%%%%%%%%%
%%%%%%%%%%%%%%%%%%%%%%%%%%%%%%%%%%%%%%%%%%%%%%%%%%%%%%%%
\begin{proof}
Let $s_0\in D^{k+1}$ and $\Sigma\subset SM$ be a section transversal
to $|\gamma_{s_0}|$ through $v_{s_0}:=\gamma_{s_0}(0)$. Let
$(V,\varphi)$
be a coordinate chart of $\Sigma$ around $v_{s_0}$ such that
$\varphi(v_{s_0})=0\in\R^{2n-2}$, $(n=dim M)$,
and such that the Poincar\'e
map $P_{s_0}$ associated to $\gamma_{s_0}$ is defined on $V$.

Set $U=\varphi(V)$ and
let $F:B_{\varepsilon}(s_0)\times U \to \R^{2n-2}$ be defined by 
\beq
F(s,x)=x- \varphi(P_{s}(\varphi^{-1}(x))).
\eeq
This map is smooth by Lemma \ref{poincare} and also
$F(s_0,0)=0$. Fixing $s_0$ and
differentiating respect to the remaining variables, we have:
\beq
DF(s_0,x)|_{x=0}=1-D\varphi(P_{s_0}(\varphi^{-1}(x)))|_{x=0}.
\eeq

Now, since the Riemannian metrics $g_s$ on $M$ are complete and pinched negatively curved, the $g_s$-geodesic flow is of Anosov type 
(\cite[Proposition 3.2]{Knieper}). This
implies that $|\gamma_{s_0}|$ is a hyperbolic periodic orbit.
Hence $DF(s_0,x)|_{x=0}$ is invertible. Consequently,
by the implicit function theorem, there exist $\varepsilon'>0$ and 
a smooth function $f:B_{\varepsilon'}(s_0)\to U$ such that
$F(s,f(s))=0$. To simplify the notation, let us take
 $\varepsilon'=\varepsilon$.

Let now
$\widetilde{\gamma}_s:\R\to SM$, $s\in B_{\varepsilon}(s_0)$ 
be the integral curve of 
of the $g_s$-geodesic field with initial condition
$\widetilde{\gamma}_s(0)=\varphi^{-1}(f(s))\in V\subset\Sigma$.
Thus, by the
theorem of continuous dependence of the solutions to ODEs on the
initial conditions and parameters (see for example \cite{sotomayor}), the map
\beq
B_{\varepsilon}(s_0)\to C^{\infty}(\R, SM)
\eeq
given by $s\mapsto\widetilde{\gamma}_s$ is continuous with respect to the
compact-open-$C^{\infty}$ topology on $C^{\infty}(\R, SM)$.

By Corollary \ref{period}, the 
prime periods $\tau_s$ of the orbits $\widetilde{\gamma}_s$
depend
continuously on $s$. We aim to reparametrize the periodic orbits 
$\widetilde{\gamma}_s$, $s\in B_{\varepsilon}(s_0)$, so that they all have the same
prime period, say $\tau$. To do this, denote by $\xi_s$ the
$g_s$-geodesic field on $SM$ and define a new family 
$\xi^*_s$ of vector
fields on $SM$ by
\beq
\xi^*_s=\frac{\tau_{s}}{\tau_{s_0}}\xi_s.
\eeq
Since both the period and the geodesic field vary continuously on 
$B_{\varepsilon}(s_0)$,
$\xi^*_s$ depends continuously on $s$ as well. It easy to verify that
the integral curves $\gamma^*_s$ of $\xi^*_s$ are just reparametrized
integral curves of the $g_s$-geodesic flow, namely
\beq
\gamma^*_s(t)=
\widetilde{\gamma}_s\left(\frac{\tau_{s}}{\tau_{s_0}}t\right)
\eeq
and that the prime period of $\gamma^*_s$ is $\tau:=\tau_{s_0}$.

Now let $S^1$ be the interval $[0,\tau]$ with the endpoints
identified. Therefore the map
\beq
B_{\varepsilon}(s_0)\to C^{\infty}(S^1, SM),
\eeq
which associates the periodic orbit $\gamma^*_s$ to the parameter
$s\in B_{\varepsilon}(s_0)$ is continuous. The curves $\alpha_s:=\tau_M\circ\gamma^*_s$ form the desired family of closed
$g_s$-geodesics.
\end{proof}
\begin{coro}\label{anosov orbits}
There exists a family of
parametrized closed smooth curves 
$\alpha_s:S^1\to M$, $s\in D^{k+1}$, such that 
\begin{itemize}
\item[i)] $[\alpha_s]=[c]$,
\item[ii)]the map $D^{k+1}\to C^{\infty}(S^1,M)$ given by
$s\mapsto\alpha_s$ is continuous, and
\item[iii)] $\alpha_s$ is a closed $g_s$-geodesic in $M$.
\end{itemize}
\end{coro}
\begin{proof}
To obtain a globally defined family of closed geodesics we consider the
subspace $\mathcal{E}\subset D^{k+1}\times SM$ defined by
%\beq
%E=\{(s,\alpha,v)\in D^{k+1}\times C^{\infty}(S^1,M)\times SM|\ 
%\alpha(S^1)=(\tau_M\circ\gamma_s)(S^1)\text{ and } v\in\gamma_s(S^1)\in SM\},
%\eeq
\beq
\mathcal{E}=\{(s,v)\in D^{k+1}\times SM|\ 
v\in |\gamma_s|\subset SM\},
\eeq
together with a projection $\pi:\mathcal{E}\to D^{k+1}$ onto the first factor, 
$\pi(s,v)=s$. 
Define $\Phi_{\varepsilon}:
B_{\varepsilon}(s_0)\times S^1\to\pi^{-1}(B_{\varepsilon}(s_0))$ by
\beq
\Phi_{\varepsilon}(s,t)=(s,\gamma_s^*(t)).
\eeq
Note that this map is a homeomorphism which defines a local trivialization
for $\pi:\mathcal{E}\to D^{k+1}$. Thus $\pi$ is a fiber bundle over a contractible
space $D^{k+1}$ so it must be a trivial bundle. Hence we can find a 
smooth map $\sigma:D^{k+1}\times S^1\to \mathcal{E}$, say 
$\sigma(s,t)=(s,\overline{\gamma}_s(t))$, with 
$|\overline{\gamma}_s|=|\gamma_s|$. This induces the desired continuous
family of closed geodesics, namely a continuous map 
$D^{k+1}\to C^{\infty}(S^1,M)$ given by 
$s\mapsto\alpha_s=\tau_M\circ\overline{\gamma}_s$, and such that
$[\alpha_s]=[c]$ for all $s\in D^{k+1}$.
\end{proof}
%%%%%%%%%%%%%%%%%%%%%%%%%%%%%%%%%%%%%%%%%%%%%%%%%%%%%%%%%%%%%%%%%%%%%%
%%%%%%%%%%%%%%%%%%%%%%%%%%%%%%%%%%%%%%%%%%%%%%%%%%%%%%%%%%%%%%%%%%%%%%
%%%%%%%%%%%%%%%%%%%%%%%%%%%%%%%%%%%%%%%%%%%%%%%%%%%%%%%%%%%%%%%%%%%%%%%%%%%%%%%%%%%%%%5
%%%%%%%%%%%%%%%%%%%%%%%%%%%%%%%%%%%%%%%%%%%%%%%%%%%%%%%%%%%%%%%%%%%%%%%%%%%%%%%%%%%%%%%
\section{Homotopy groups of $MET^{<0}_{\infty}(M,g)$}\label{components}
The proof of Theorem \ref{main1} follows closely the method
of proof of the Main Theorem of \cite{farrell2010}. We will sketch their proof here indicating 
how to adapt it to the case of noncompact manifolds under the conditions on the
curvature and behavior at infinity specified in the statement of Theorem \ref{main1}.

Recall that we are assuming $MET^{<0}(M)$ is nonempty. For any 
$g\in MET^{<0}(M)$ we will construct, for certain values
of $k$ depending on the
dimension of $M$, continuous maps
$S^{k}\to MET_{\infty}^{<0}(M)$ representing nontrivial classes
in $\pi_kMET_{\infty}^{<0}(M,g)$.

We begin by recalling the following facts related to the topology of
the group of stable diffeomorphisms of the circle. Denote by
$Diff(S^1\times S^{n-2}\times I,\partial)$ 
the space (with the $C^{\infty}$-topology)
of smooth diffeomorphisms of $S^1\times S^{n-2}\times I$
whose restriction to $S^1\times S^{n-2}\times\{0,1\}$
is the identity map.
$P^{\text{TOP}}(S^1\times S^{n-2})$ denotes the space of topological
pseudoisotopies of $S^1\times S^{n-2}$, i.e. it is the space of all
homeomorphisms of $S^1\times S^{n-2}\times I$ that fix 
$S^1\times S^{n-2}\times\{0\}$ pointwise.
\begin{thm}\label{diffeos}
Let $I$ denote closed interval and $S^m$ denote an $m$-dimensional sphere. Then
\begin{enumerate}
\item[1)] If $n>9$, $(\Z/2)^{\infty}\subseteq\pi_0 Diff(S^1\times S^{n-2}\times I,\partial)$.
\item[2)] If $n>13$, $(\Z/2)^{\infty}\subseteq\pi_1 Diff(S^1\times S^{n-2}\times I,\partial)$.
\item[3)] If $\frac{n-10}{2}>2p-4$, $(\Z/p)^{\infty}\subseteq\pi_{2p-4} Diff(S^1\times S^{n-2}\times I,\partial)$.
\end{enumerate}
Moreover the inclusion map 
$Diff(S^1\times S^{n-2}\times I,\partial)\hookrightarrow P^{\text{TOP}}(S^1\times S^{n-2})$
is $\pi_k$-injective for $n$ and $k$ as in $1)$, $2)$ and $3)$, when
restricted to the subgroups $(\Z/2)^{\infty}$, $(\Z/2)^{\infty}$ and $(\Z/p)^{\infty}$
respectively.
\end{thm}
\begin{rmk}
The proof of this theorem (\cite[Section 4]{farrell2010})
follows from combining the natural involution in the
space of pseuodisotopies of a manifold with the results of Hatcher
\cite{hatcher1973} and Igusa \cite{igusa1982} (for item 1)), Igusa
\cite{igusa1982} (for item 2)), and Waldhausen \cite{waldhausen1978}
and \cite{grunewald08} (for item 3). 
\end{rmk}
\begin{proof}[of Theorem \ref{main1}]
We prove 1), 2) and 3) simultaneously. 
Recall that by Proposition \ref{existence} and Corollary \ref{normalbundle}, 
there exist a free homotopy class of loops $[c]\in[S^1,M]$ 
that can be represented by a closed
$g$-geodesic $\alpha:S^1\to (M,g)$ whose
normal bundle is trivial. Additionally,
suppose that the normal
injectivity radius of $\alpha$ is $2r$ so that we identify a tubular
neighborhood of 
$\alpha$ of radius $2r$, with $S^1\times\D^{n-1}_{2r}$,
where $2r$ denotes a closed
disk of radius $2r$. Let now 
$\varphi\in Diff(S^1\times S^{n-2}\times [r,2r],\partial)$.
Clearly, $\varphi$
can be extended to a self-diffeomorphism of $M$ by taking the identity
map outside
$S^1\times S^{n-2}\times [r,2r]\subset S^1\times \D_{2r}^{n-1}\subset M$. We
keep denoting this extension by $\varphi$. Now, on $M$ we define a
new Riemannian metric, denoted by $\varphi g$, by declaring the 
map $\varphi:(M,g)\to (M,\varphi g)$ and isometry. This new metric
is clearly and element in $MET^{<0}_{\infty}(M,g)$. Note that by
this procedure we have defined a map
\beq
\Phi:Diff(S^1\times S^{n-2}\times I,\partial)
\to
MET^{<0}_{\infty}(M,g),
\eeq
by $\Phi(\varphi)=\varphi g$.

If, according to Theorem \ref{diffeos},
$f:S^k\to Diff(S^1\times S^{n-2}\times I,\partial)$ represents a
nontrivial class corresponding to either
$(\Z/2)^{\infty}\subseteq\pi_0 Diff(S^1\times S^{n-2}\times I,\partial)$ if $n>9$,
or $(\Z/2)^{\infty}\subseteq\pi_1 Diff(S^1\times S^{n-2}\times I,\partial)$ if $n>13$,
or $(\Z/p)^{\infty}\subseteq\pi_{2p-4} Diff(S^1\times S^{n-2}\times I,\partial)$ if
$\frac{n-10}{2}>2p-4$;
then the proof of the next claim completes the proof of Theorem \ref{main1}.\\
\begin{claim}
The composite
\beq
S^k\xrightarrow{f} Diff(S^1\times S^{n-2}\times I,\partial)
\xrightarrow{\Phi}MET^{<0}_{\infty}(M,g)
\eeq
represents a non-trivial class in $\pi_kMET^{<0}_{\infty}(M,g)$.
\end{claim}
We prove the claim by contradiction. Assume that there exists a 
continuous map $D^{k+1}\to MET^{<0}_{\infty}(M,g)$,
$s\mapsto g_s$, that extends 
$\Phi\circ f$. We will construct a continuous map 
$F:D^{k+1}\to P^{\text{TOP}}(S^1\times S^{n-2})$ such that 
$F|_{S^k}=f$, contradicting the fact that $f$ was chosen to represent 
a non-trivial element in $\pi_kDiff(S^1\times S^{n-2}\times I,\partial)$ mapped
monomorphically into $\pi_kP^{\text{TOP}}(S^1\times S^{n-2})$.
Recall that by Corollary \ref{anosov orbits}, the $g_s$-geodesics
$\alpha_s:S^1\to M$ representing the class $[\alpha]=[c]\in [S^1,M]$,
can be assumed to depend continuously on the parameter $s$.

\par Let $p:\widehat{M}\to M$ be the covering space of 
$M$ corresponding to the infinite
cyclic subgroup of $\pi_1(M,\alpha(1))$ generated by $\alpha$. 
For each 
$s\in D^{k+1}$, $\widehat{M}$ can be given a Riemmanian metric $\hat{g}_s$ by
pulling-back along $p$, i.e.
$\hat{g}_s=p^*g_s$. The immersed geodesic
$\alpha$ has a lifting to an embedded closed geodesic 
$\widehat{\alpha}:S^1\to\widehat{M}$ in $\widehat{M}$.

Moreover, each metric $g_s$ can be lifted to $\widehat{M}$ and the
$g_s$-geodesics $\alpha_s$ have liftings to embedded 
$\hat{g}_s$-geodesics in
$\widehat{M}$ ($\hat{g}_s$ is the pullback of $g_s$). The metrics $\hat{g}_s$ clearly vary continuously with $s$, and in fact we \textit{assume} that $\widehat{\alpha}_s=\widehat{\alpha}$ for all $s\in D^{k+1}$. This
assumption can be made due to the following theorem proven in 
\cite[Lemma 1.4]{farrell2010}.
\begin{thm}\label{hofiber}
The homotopy fiber of the inclusion map 
$Emb(S^1,\widehat{M})\hookrightarrow C^{\infty}(S^1,\widehat{M})$ is
$(n-5)$-connected, where $Emb(S^1,\widehat{M})$ and 
$C^{\infty}(S^1,\widehat{M})$ are the spaces of smooth embeddings 
and smooth functions from $S^1$ to $\widehat{M}$ respectively.
\end{thm}
\begin{rmk}
In other words, every continuous $(k+1)$-parameter family of 
embeddings of $S^1$ into $\widehat{M}$ can be deformed
through smooth maps into a given embedding $\alpha$. Then the 
assumption becomes true when one redefines the metrics according to this
deformation (a pull-back along an isotopy). See 
\cite[Claim 1, p.292]{farrell2010} for more details. Note, though, that
in order to prove Theorem \ref{hofiber} it is necessary that the closed geodesics vary 
continuously as \textit{parametrized} maps, rather than just as point sets (orbits).
That is exactly what we achieved in Section \ref{preliminaries}.
\end{rmk}

If we lift the tubular neighborhood $S^1\times\D^{n-1}_{2r}$ to $\widehat{M}$
we obtain a countable number of connected components:
\beq 
p^{-1}(S^1\times\D^{n-1}_{2r})=\displaystyle\coprod_{i=0}^{\infty} C_i
\eeq
where exactly one of the $C_i$'s, say $C_0$, is homeomorphic to
$S^1\times\D^{n-1}_{2r}$
and $C_i$, $i\neq 0$ is homeomorphic to $\R\times\D^{n-1}_{2r}$. 

Additionally, we identify $\widehat{M}$ with $S^1\times\R^{n-1}$, via a 
homeomorphism $\mu:S^1\times\R^{n-1}\to\widehat{M}$ obtained 
by restricting the exponential map 
(respect to the metric $g$) to the normal bundle of $\alpha$. This is made in a way
that the liftings $\widehat{\alpha}_s$ of $\alpha_s$ to $\widehat{M}$ all coincide
with $S^1\times\{0\}\subset C_0$ and the rays 
$\{z\}\times\R^+v\subset S^1\times S^{n-2}\times (0,\infty)\subset S^1\times\R^{n-1}$, 
correspond to $\hat{g}$-geodesic rays in $\widehat{M}$ normal to $\widehat{\alpha}$.

Once these identifications have been made, we define a $(k+1)$-parameter
family of self homeomorphisms of $S^1\times\R^{n-1}$, 
$F_s:S^1\times\R^{n-1}\to S^1\times\R^{n-1}$ in the following way:

First notice that by \cite[Lemma 1.6, Claim 2]{farrell2010}, there is a small
$\delta>0$ and a closed neighborhood $W\subset\widehat{M}$ 
of $\widehat{\alpha}(S^1)$ (with boundary diffeomorphic to $S^1\times S^{n-2}$)
such that, for all $s\in D^{k+1}$ and all
$(z,v)\in S^1\times S^{n-2}\subset S^1\times\R^{n-1}$, there exists a unique
$\widehat{g}_s$-geodesic ray $w^s_{(z',v)}:[0,\infty)\to\widehat{M}$ emanating
$\widehat{g}_s$-perpendicularly from some point 
$z'\in\widehat{\alpha}(S^1)=\widehat{\alpha}_s(S^1)$ and intersecting
the boundary $\partial W$ of $W$ transversally 
only once at $\mu(z,\delta v)$. 
Thus we can define a $(k+1)$-parameter family of self-homeomorphisms of
$S^1\times S^{n-2}\times [\delta,\infty)$ by
\beq
F_s(z,v,t)=\mu^{-1}\left(w^s_{(z',v)}\left(\frac{\delta'}{\delta}\, t\right)\right),
\eeq
where $w^s_{(z',v)}(\delta')=\mu(z,\delta v)$. 

It follows from \cite[pag.293-294]{farrell2010} that this family 
is continuous on $s\in D^{k+1}$. The reason is that the map $F_s$ is
essentially a composition of a trivialization of the 
normal bundle of $\widehat{\alpha}$ with the normal
(to  $\widehat{\alpha}$) exponential map respect to $\widehat{g}_s$, and 
those two maps depend continuosly on the metric.

Figure 3 below helps visualizing this family of self-homeomorphisms.

%For all $(z,x)\in S^1\times\R^{n-1}$, there exists a unique $g$-geodesic
%ray $c_{z}:[0,\infty)\to S^1\times\R^{n-1}$ 
%emanating perpendicularly from $S^1$ passing through $x$. Let $t_0\in\R^+$
%defined by $c_{z}(t_0)=x$.
%For each
%$s\in D^{k+1}$, \cita shoot'' $g_s$-geodesic rays 
%$c^s_{z}:[0,\infty)\to S^1\times\R^{n-1}$ emanating perpendicularly (respect to $g_s$) from $(z,0)\in S^1\times \{0\}$. Then 
%\beq
%\bar{F}_s(z,x)=c^s_z(t_0)
%\eeq

%Notice that by \cite[Lemma 1.6, Claim 2]{farrell2010}, there
%is a small positive number $\delta$ such that for each $s\in D^{k+1}$ and
%$(z,v)\in S^1\times S^{n-2}\subset S^1\times\R^{n-1}$,
%there exits a unique $\hat{g}_s$-geodesic ray 
%$w^s_{z'}:[0,\infty)\to S^1\times\R^{n-1}$ emanating perpendicularly from
%$z'\in\alpha$ intersecting
%$S^1\times S^{n-2}\times\{\delta\}$ only once at $(z,v,\delta)$.
%Thus we can define a continuous $(k+1)$-parameter family of continuous
%maps
%$F_s:S^1\times S^{n-2}\times[\delta,\infty)\to S^1\times S^{n-2}\times[\delta,\infty)$ by
%\beq
%F_s(z,v,t)=w_{z'}^s\left(\frac{\delta}{\delta'} t\right)
%\eeq
%where $w_{z'}^s(\delta')=(z,v,\delta)$

%The Figure 3. below illustrates the definition of $F_s$.

In order to simplify the notation assume that $\delta=1$. The space 
$\widehat{M}$ can be compactified by adding points at infinity corresponding to asymptotic $g$-quasi-geodesics rays emanating
perpendicularly from $\widehat{\alpha}$. This boundary is denoted by
$\partial_{\infty}\widehat{M}$. (See \cite[Section 2]{farrell2010}.)

\begin{figure}[!h]
\vspace*{0cm}
    \begin{center}
    \includegraphics[scale=0.5]{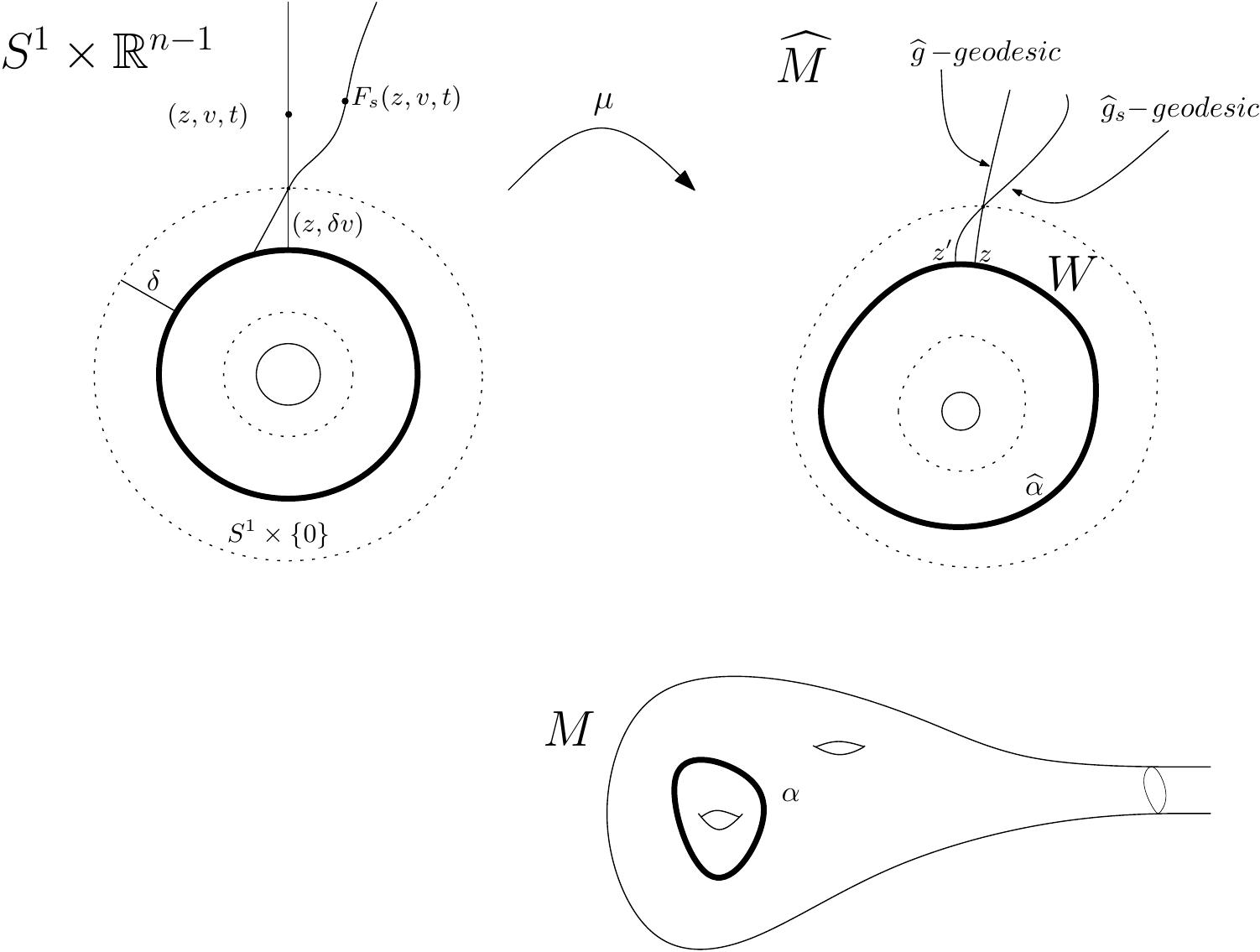}
    \caption{Self-homeomorphisms giving rise to an isotopy}
    \end{center}
\end{figure}

Thus we can extend the homeomorphisms $F_s$ to self-homeomorphisms
of $S^1\times S^{n-2}\times[1,\infty]$  by
\beq
F_s(z,v,\infty)=\mu^{-1}([w^s_{(z',v)}]),
\eeq
where $[w^s_{(z',v)}]\in\partial_{\infty}\widehat{M}$ denotes the class 
of the ray $w^s_{(z',v)}:[0,\infty)\to\widehat{M}$.
For this extension to be well defined 
we need to guarantee
that the $\widehat{g}_s$-geodesic rays emanating from $\widehat{\alpha}$ are, indeed,
$\widehat{g}$-quasi-geodesic rays. From the results of 
\cite[Section 2]{farrell2010}, it is enough to prove that the 
identity map 
$(\widetilde{M},\widetilde{g})\to (\widetilde{M},\widetilde{g}_s)$
is a quasi-isometry for all $s\in D^{k+1}$. 
This is what we show in the next lemma.

\begin{lemma}\label{identity}
Let $M$ be a smooth manifold such that $MET^{<0}(M)\neq\emptyset$. Let
$g\in MET^{<0}(M)$ and let $D^{k+1}\to MET^{<0}_{\infty}(M,g)$
be a continuous map.
If $\widetilde{g}_s$ denotes the pull-back metric of 
$g_s$ along the universal covering projection $p:\widetilde{M}\to M$, then
the identity map
\beq
id:(\widetilde{M},\widetilde{g})\to(\widetilde{M},\widetilde{g}_s)
\eeq
is a quasi-isometry for all $s\in D^{k+1}$, where $g_s$ is the image of
$s$ in $MET^{<0}_{\infty}(M,g)$.
\end{lemma}
\begin{proof}
%Without lost of generality we can assume that the neighborhoods of the ends
%of $(M,g_s)$ are of the form 
%$M^s_{<\epsilon}:=\{x\in M|\text{InjRad}_{g_s}(x)<\epsilon\}$.
By continuity, the subset 
$\{g_s|s\in D^{k+1}\}\subset MET^{<0}_{\infty}(M,g)$ is compact.
Since $MET^{<0}_{\infty}(M,g)$ is given
the direct limit topology, we can take a sufficiently large compact set
$K\subset M$ such that 
\beq
(M-K,g)\xrightarrow{id} (M-K,g_s)
\eeq 
is an isometry for
all $s\in D^{k+1}$. 

Let $S_gM$ be the unit tangent bundle of $M$ with respect to $g$ and let
$i:K\hookrightarrow M$ be the inclusion map into $M$.
Note that since $g_s$ is a 
$C^{\infty}$-continuously varying family of metrics, the map
\beq
D^{k+1}\times i^*(S_gM)\to (0,\infty)
\eeq
which sends a pair $(s,v)$ with $g(v,v)=1$ and $\tau_M(v)\in K$
to $g_s(v,v)$ is smooth. Here $i^*(S_gM)$ denotes the restiction of the sphere
bundle of $M$ to $K$. The domain of this function is compact, hence there
exist constants $a,b>0$ such that $a\leq g_s(v,v)\leq b$.
\par Now for $x,y\in p^{-1}(K)$, let $\gamma:[0,1]\to\widetilde{M}$ 
be a continuous
path in $\widetilde{M}$ with $\gamma(0)=x$ and $\gamma(1)=y$ realizing the 
$\tilde{g}$-distance between $x$ and $y$, which we denote by $d_{\tilde{g}}$.
Then $d_{\tilde{g}_{s}}(x,y)\leq\ell_{\tilde{g}_{s}}(\gamma)\leq
\sqrt{b}\, d_{\tilde{g}}(x,y)$, where
$\ell_{\tilde{g}_s}$ denotes the length function induced by
the Riemannian metric $\tilde{g}_s$.
Similarly one proves that $d_{\tilde{g}}(x,y)\leq\frac{1}{\sqrt{a}}\, 
d_{\tilde{g}_{s}}(x,y)$

As for points that do not project down to $K$, 
we only have to remark again that
the identity map $id:(\widetilde{M},\tilde{g})\to
(\widetilde{M},\tilde{g}_s)$ is an
isometry, when restricted to $p^{-1}(M-K)$.
\par If we let $\lambda=\text{max}\{\frac{1}{\sqrt{a}},\sqrt{b}\}$, 
it now follows that for all $x,y\in\widetilde{M}$ and 
$s\in D^{k+1}$
\beq
\frac{1}{\lambda}\, d_{\tilde{g}}(x,y)\leq d_{\tilde{g}_{s}}(x,y)\leq\lambda\, d_{\tilde{g}}(x,y).
\eeq
This completes the proof of the lemma. \\
\end{proof}
\begin{rmk}
Note that Lemma \ref{identity} is the only place where the control at
infinity is used. If no condition on the ends is imposed,
one may not
obtain that the identity map on the universal cover of $M$ is a
quasi-isometry, and the maps
$F_s$ don't seem to extend to the points at infinity in an obvious manner.
\end{rmk}
From this lemma it follows that the $(k+1)$-parameter family
of self-homeomorphisms $F_s:S^1\times S^{n-2}\times [1,\infty]$ is well
defined.
Moreover, in \cite[Section 3]{farrell2010}, it is proven that $F_s$ 
is indeed a continuously varying $(k+1)$-parameter family of 
self-homeomorphisms
of $S^1\times S^{n-2}\times [1,\infty]$; and one notices that $F_s$
restricts to the identity on $S^1\times S^{n-2}\times\{1\}$ for all $s$.
In other
words the map $D^{k+1}\to P^{\text{TOP}}(S^{1}\times S^{n-2})$ which sends
$s$ to $F_s$ is continuous. Finally, it is checked in 
\cite[Section 3, Claims 5 and 6]{farrell2010}
that $F_s\sim f(s)$ for $s\in S^k=\partial D^{k+1}$. This gives the
desired contradiction and proves the claim. \\
\end{proof}
%%%%%%%%%%%%%%%%%%%%%%%%%%%%%%%%%%%%%%%%%%%%%%%%%%%%%%%%%%%%%%%
%%%%%%%%%%%%%%%%%%%%%%%%%%%%%%%%%%%%%%%%%%%%%%%%%%%%%%%%%%%%%%%
\par The proof of Theorem \ref{main1} yields the following addendum.
\begin{add}\label{hyper}
Let $(M,h_0)$ be a hyperbolic manifold with finite volume and 
$dim\, M\geq 10$. Then $M$ admits a hyperbolic metric $h_1$ of finite
volume such that: \textit{i)} $h_1\in MET^{<0}_{\infty}(M,h_0)$, and
\textit{ii)}
$h_0$ can't be joined to $h_1$ by a continuous path of metrics in 
$MET^{<0}_{\infty}(M,h_0)$.
\end{add}
\begin{proof}
If $h_0$ is a hyperbolic metric on $M$ then for any 
$\varphi\in Diff(S^1\times S^{n-2}\times I)$ as in the proof of 
Theorem \ref{main1}, the hyperbolic metric
$h_1:=\varphi h_0$ clearly satisfies $i)$ and $ii)$.\\
\end{proof}
%%%%%%%%%%%%%%%%%%%%%%%%%%%%%%%%%%%%%%%%%%%%%%%%%%%%%%%%%%%%%%%%%%%%%%%%%%%%%%%%%%%%%%%
%\newpage
%\section{The topology of the space of nonpositively curved metrics on $DM$}.
 \bibliographystyle{alpha} % style for bibliography
 {\footnotesize
 \bibliography{biblio}}     % biblio.bib is the name of our database
\Addresses
\end{document}